\documentclass[preprint,11pt]{article}

\usepackage{amsmath,amsfonts,amssymb,amsthm,comment,cases,accents}

\newcommand{\R}{{\mathbb R}}

\theoremstyle{plain}
\newtheorem{Theorem}{Theorem}[section]

\newtheorem{lemma}[Theorem]{Lemma}
\newtheorem{Proposition}[Theorem]{Proposition}

\newenvironment{proposition}{\begin{Proposition} }{\end{Proposition}}
\newenvironment{Proof}{\begin{proof} }{\end{proof}}
\newenvironment{theorem}{\begin{Theorem} }{\end{Theorem}}

\theoremstyle{definition}
\newtheorem{Remark}[Theorem]{Remark}
\newenvironment{remark}{\begin{Remark} \rm}{\end{Remark}}

\def\dist{{\rm dist}\, }

\def\al{{\alpha}}
\def\D{{\Delta}}

\def\p{{\varphi}}
\def\o{{\omega}}
\def\O{{\Omega}}

\def\n{{\mathbf n}}
\def\bfA{{\mathbf A}}
\def\bfE{{\mathbf E}}
\def\bfH{{\mathbf H}}

\newcommand{\oveta}{{\overline \eta}_u}
\newcommand{\ovxi}{{\overline \xi}_u}
\newcommand{\ovtheta}{{\overline \theta}_u}

\begin{document}

\title{Standing Waves for Nonautonomous Klein-Gordon-Maxwell Systems}
\author{Monica Lazzo \thanks{
 Dipartimento di Matematica, Universit\`a degli Studi di Bari Aldo Moro, via E.~Orabona 4, 70125 Bari, Italy; e-mail: monica.lazzo@uniba.it,
 lorenzo.pisani@uniba.it} \and Lorenzo Pisani \footnotemark[1]}
 \date{}
\maketitle

  \abstract{We study a Klein-Gordon-Maxwell system, in a bounded spatial domain, under Neumann boundary conditions on the electric potential. We allow a nonconstant coupling coefficient. For sufficiently small data, we find infinitely many standing waves. \\[2mm]
  }

\section{Introduction}\label{intro}
We are interested in the system of nonautonomous elliptic equations
\begin{equation}\label{KGM}
\begin{alignedat}{2}
\Delta u  &= m^2 \, u - \bigl(\o + q(x) \, \phi\bigr)^2 u   \qquad & \hbox{in $\O$,} \\[1mm]
\Delta \phi &= q(x) \, \bigl(\o+ q(x) \, \phi\bigr) \, u^2 & \hbox{in $\O$.} \end{alignedat}
\end{equation}
where $\D$ is the Laplace operator in $\R^3$, $\O \subset \R^3$ is a bounded and  smooth domain, $m, \o \in \R$,  $q \in L^6(\O) \setminus\{0\}$. We complement these equations with the boundary conditions
\begin{subequations}\label{BC-no}
\begin{alignat}{2}
u  &= 0  \qquad & \hbox{on $\partial \O$,} \label{BC-a} \\
\dfrac{\partial \phi}{\partial \n}  &= \al & \hbox{on $\partial \O$,}
\label{BC-b}
\end{alignat}
\end{subequations}
where $\n$ is the unit outward normal vector to $\partial \O$ and $\al \in H^{1/2}(\partial \O)$.

We look for {\em nontrivial} solutions, by which we mean pairs $(u,\phi) \in H^1_0(\O) \times H^1(\O)$, satisfying~\eqref{KGM}-\eqref{BC-no} in the usual weak sense, with $u\ne0$. Note that, if $(u,\phi)$ is a nontrivial solution, the pair $(-u,\phi)$ is a nontrivial solution  as well.

System~\eqref{KGM} arises in connection with the so-called Klein-Gordon-Maxwell equations, which model the interaction of a charged matter field with the electromagnetic field $(\bfE,\bfH)$.
They are the Euler-Lagrange equations of the Lagrangian density
\begin{equation*}
\begin{split}
{\cal L}_{KGM} & = \tfrac{1}{2} \bigl( |(\partial_t + i \, q \, \phi) \, \psi|^2 - |(\nabla - i \, q \, \bfA) \, \psi|^2 - m^2 \, |\psi|^2 \bigr) + {} \\[2mm]
& + \tfrac{1}{8\pi} \bigl(|
\nabla \phi + \partial_t \bfA|^2 - |\nabla \times \bfA|^2 \bigr) \, ,
\end{split}
\end{equation*}
where $\psi$ is a complex-valued function representing the matter field, while $\phi$ and $\bfA$ are the gauge potentials, related to the electromagnetic field via the equations
$\bfE = -\nabla \phi - \partial_t \bfA$, $\bfH = \nabla \times \bfA$.
For the derivation of the Lagrangian density, and details on the physical model, we refer to~\cite{BF-2014, bleeker, felsager}.

Confining attention to {\em standing waves}, in equilibrium with a purely electrostatic field, amounts to imposing $\psi(t,x) = e^{i \o t} \, u(x)$, where $u$ is a real-valued function and $\o$ is a real number, $\bfA = 0$, and $\phi=\phi(x)$. With  these choices, the Klein-Gordon-Maxwell equations considerably simplify and become~\eqref{KGM}.
In the physical model, the boundary condition~\eqref{BC-a} means that the matter field is confined to the region $\O$, while~\eqref{BC-b} amounts to prescribing the normal component of the electric field on $\partial \O$. Up to a sign, the surface integral $\int_{\partial \O} \al \, d\sigma$ represents the flux of the electric field through the boundary of $\O$, and thus, the total charge contained in $\O$.

Let us point out that in the physical model the coupling coefficient $q$ is constant (see~\cite[Subsection~5.15]{BF-2014});
nonconstant coefficients, however, are worth investigating from a mathematical viewpoint.

For a constant coupling coefficient $q$, existence results for nontrivial solutions to Problem~\eqref{KGM}-\eqref{BC-no} were obtained in~\cite{DPS-2010}.
In this case, an invariance property holds and solutions to~\eqref{KGM}-\eqref{BC-no}, for arbitrary $\o$, correspond to solutions of the same system with $\o=0$ (that is, {\em static} solutions). Thus, with no loss of generality, in~\cite{DPS-2010} the authors confined their attention to static solutions.
Their results were generalized to the nonautonomous case  in~\cite{DPS-2014}, assuming that the coupling coefficient $q$ vanishes at most on a set of measure zero; this restriction on the zero-level set of $q$ was later removed in~\cite{lp1}.
Note that, absent the invariance property,
the existence of solutions to~\eqref{KGM}-\eqref{BC-no} with $\o\ne 0$ does not follow from the results in~\cite{DPS-2014, lp1}. Investigating Problem~\eqref{KGM}-\eqref{BC-no} with $\o \ne 0$ is precisely the goal of the present paper.

\begin{theorem}\label{main1}
Assume $\int_{\partial \O} \al \, d\sigma \ne 0$.  Suppose that $|\o| \le |m|$ and the function~$q$ satisfies the following condition: \\[2mm]
{\rm (Q)} \ there exists $q_0 \in (0,+\infty)$ such that $\bigl|\{x\in \O \, | \, 0<|q(x)|<q_0\}\bigr|=0$. \\[2mm]
Then, if
$\|\al\|_{H^{1/2}(\partial \O)} \, \|q\|_{L^6(\O)}$ is sufficiently small,
Problem~\eqref{KGM}-\eqref{BC-no} has a sequence $\{(u_n,\phi_n)\}$ of distinct nontrivial solutions with the following properties:
\begin{enumerate}
\item[{\rm (i)}] $u_0 \ge 0$ in $\O$;
\item[{\rm (ii)}] every bounded subsequence $\{u_{k_n}\}$ satisfies $\|q \, u_{k_n}\|_{L^3(\O)} \to 0$ as $n \to \infty$.
\end{enumerate}
\end{theorem}

A function $q$ that satisfies~(Q) may vanish in $\O$, even on a set of positive measure; however, where $q$ does not vanish, it must be bounded away from zero.
This condition appears in results on the closedness of the range of the multiplication operator $u \mapsto q \, u$ (see~\cite{ramos}).

Without assumption~(Q), we find nontrivial solutions provided that $\o$ varies in a smaller range.

\begin{theorem}\label{main2}
Assume $\int_{\partial \O} \al \, d\sigma \ne 0$. Suppose that $|\o| \le |m|/\sqrt 2$. Then the same conclusions as in Theorem~\ref{main1} hold.
\end{theorem}

Following the approach in~\cite{lp1}, we apply Ljusternik-Schnirelmann theory to a functional $J$, defined in an open subset $\Lambda_q$ of $H^1_0(\O)$, whose critical points correspond to nontrivial solutions to Problem~\eqref{KGM}-\eqref{BC-no}. Compared to the functional considered in~\cite{lp1}, here $J$~contains several additional terms, which depend on $\o$. Assuming $|\o| \le |m|$, all but one of these additional terms can be easily dealt with and entail no major complications in the study of~$J$.
Under assumption~(Q), the exceptional term (the third summand in~\eqref{J-dec} below) can be controlled in a uniform fashion (see Lemma~\ref{theta-bounded}).
Without assumption~(Q), uniform bounds on the exceptional term are not available (see Remark~\ref{noQ}). However, $J$ retains its main properties for smaller values of $\o$, as in Theorem~\ref{main2}.

If the data are as small as in Theorem~\ref{main1},
the condition $\int_{\partial \O} \al \, d\sigma \ne 0$
is necessary for the existence of nontrivial solutions, as in~\cite{DPS-2010, DPS-2014, lp1}.

\begin{theorem}\label{main3}
Assume that $|\o| \le |m|$ and $\|\al\|_{H^{1/2}(\partial \O)} \, \|q\|_{L^6(\O)}$ is as small as required in Theorem~\ref{main1}. If $\int_{\partial \O} \al \, d\sigma = 0$, then Problem~\eqref{KGM}-\eqref{BC-no} has no nontrivial solutions.
\end{theorem}

The assumptions in Theorems~\ref{main1}-\ref{main3} are consistent with the literature, albeit on problems in unbounded domains.
Limitations on the range of $\omega$ already appeared in~\cite{strauss}; later, they were required in~\cite{BF-2002} and the subsequent stream of related papers. Conditions on the smalless of $q$ (if $\alpha$ is fixed) were imposed, for instance, in \cite[Theorem 104]{BF-2014}, in line with Coleman's conjecture in~\cite{coleman}.

We conclude this section by mentioning some recent work, loosely related to our own. In addition to the papers cited in~\cite{lp1}, we refer to~\cite{chen-tang,chen-li} for results on Klein-Gordon-Maxwell systems in $\R^3$.
We also refer to~\cite{miya-moura-ruv} for a variant of the system involving fractional operators; to~\cite{clapp-ghim-mich,dav-med-pomp} for results on Klein-Gordon-Maxwell-Proca systems; to~\cite{bon-dav-pomp, chen-song} for Klein-Gordon systems coupled with Born-Infeld type equations.

The paper is organized as follows. In Section~\ref{prelims} we collect some preliminary results and introduce the set $\Lambda_q$. In Section~\ref{J-birth} we define the functional~$J$ and decompose it into the sum of several components, which we analyze separately. In Section~\ref{prop-of-J} we show that $J$ satisfies the requirements in Ljusternik-Schnirelmann theory. Finally, in Section~\ref{proof-of-main} we prove Theorems~\ref{main1}-\ref{main3}.

\section{Preliminaries}\label{prelims}

Throughout the paper we will use the following notation:
\begin{itemize}
\item For any integrable function $f:\O\to \R$, $\overline f$ is the average of $f$ in $\O$ and $\|f\|_p$ is the usual norm in $L^p(\O)$ ($p\in[1,\infty]$);
\item $H^1_0(\O)$ is endowed with the norm $\|\nabla f\|_2$;
\item $H^1(\O)$ is endowed with the norm $\|f\|:= \left(\|\nabla f\|_2^2 + \left|\overline f\right|^2\right)^{1/2}$;
\item for $p\in(1,6]$, $\sigma_p$ is the smallest positive number such that
    $\|f\|_p \le \sigma_p \, \|\nabla f\|_2$ for every $f\in H^1_0(\O)$;
\item for $p\in(1,6]$, $\tau_p$ is the smallest positive number such that
    $\|f\|_p \le \tau_p \, \|f\|$ for every $f\in H^1(\O)$;
\item $A:= \int_{\partial \O} \al \, d\sigma$, \ $\|\al\|_{1/2} := \|\al\|_{H^{1/2}(\partial \O)}$.
\end{itemize}

\subsection{Reduction to homogeneous boundary conditions}

As in~\cite{lp1}, we begin by turning Problem \eqref{KGM}-\eqref{BC-no} into an equivalent problem with homogeneous boundary conditions in both variables.

Let $\chi \in H^2(\O)$ be the unique solution of
\begin{equation}\label{chi}
\Delta \chi = \dfrac{A}{|\O|} \quad \hbox{in $\O$}\, , \qquad
 \dfrac{\partial \chi}{\partial \n} = \al \quad \hbox{on $\partial \O$}\, , \qquad
\int_\O \chi \, dx = 0 \, .
\end{equation}
Note that, by elliptic regularity theory and Sobolev's inequalities, there exists $\kappa \in(0,\infty)$ such that
\begin{equation}\label{chi-infty}
\|\chi\|_{\infty} \le \kappa \, \|\al\|_{1/2} \, .
\end{equation}
With $\varphi:= \phi-\chi$, Problem~\eqref{KGM}-\eqref{BC-no} is equivalent to
\begin{equation}\label{PROB}
\begin{cases}
\Delta u = m^2 u - \bigl(\o + q \, (\p+\chi)\bigr)^2 \, u \quad & \hbox{in $\O$,} \\[1mm]
\Delta \p =  q \, \bigl(\o + q \, (\p+\chi)\bigr) \, u^2 -  \dfrac{A}{|\O|} & \hbox{in $\O$,} \\[1mm]
\hskip 3.6mm u = \dfrac{\partial \p}{\partial \n} = 0 & \hbox{on $\partial \O$.}
\end{cases}
\end{equation}

Weak solutions of~\eqref{PROB} correspond to critical points of
the functional $F$ defined in $H^1_0(\O) \times H^1(\O)$ by
\begin{equation*}
F(u,\p) =  \|\nabla u\|_2^2 +  \int_\O \bigl(m^2 - \bigl(\o + q\, (\p+\chi)\bigr)^2 \bigr) u^2 \, dx - \|\nabla \p\|_2^2 + 2 \, A \, \overline \p \, .
\end{equation*}

Indeed, standard computations show that $F$ is continuously differentiable in $H^1_0(\O) \times H^1(\Omega)$ with
\begin{align*}
\langle F'_u(u,\p) , v \rangle  &= 2 \int_\O \Bigl( \nabla u \nabla v  + \bigl(m^2 - \bigl(\o + q\, (\p+\chi)\bigr)^2 \bigr) \, u \, v \Bigr) \, dx \, , \\[1mm]
\langle F'_\p(u,\p) , \psi \rangle  & =
- 2 \int_\O \Bigl( \nabla \p \nabla \psi  + \bigl( q \, \bigl(\o + q\, (\p+\chi)\bigr) u^2  - \dfrac{A}{|\O|} \bigr) \, \psi \Bigr)  \, dx \, ,
\end{align*}
for every $u, v \in  H^1_0(\O)$ and $\p, \psi\in H^1(\O)$.
 Since $F$ is unbounded from above and from below, even modulo compact perturbations, a straightforward application of well-known results in critical point theory is precluded.

We follow~\cite{BF-2002} and associate solutions to Problem~\eqref{PROB} with critical points of a functional~$J$ that depends only on the variable $u$ and falls within the scope of classical critical point theory.
The main ingredient in the construction of $J$ is solving for $\p$ the second equation in~\eqref{PROB}. We will repeatedly apply the following result.

\begin{proposition}\label{equation}
For $b \in L^3(\O) \setminus \{0\}$ and $\rho\in L^{6/5}(\O)$, the homogeneous Neumann problem associated with the equation
\begin{equation}\label{aux-eq}
- \Delta \p + b^2 \, \p = \rho \,
\end{equation}
has a unique solution ${\cal L}_{b}(\rho)$ in $H^1(\O)$.
If $\rho$ does not change sign in $\O$, then $\rho \, {\cal L}_{b}(\rho) \ge 0$ in $\O$. Furthermore, ${\cal L}_{b}(\rho)$ depends continuously on $b$ and $\rho$.
\end{proposition}
\begin{proof}
See~\cite[Proposition 2.1, Remark 2.2, and Remark 2.3]{lp1}.
\end{proof}

\begin{remark}\label{uniform-bound}
Let $b \in L^3(\O) \setminus \{0\}$ and $h \in L^\infty(\O)$.
Observe that, for any $\tau \in \R$, ${\cal L}_{b}(b^2 \, (h+\tau)) = {\cal L}_{b}(b^2 \, h) + \tau$.
With $\tau = -\inf h$ and $\tau = -\sup h$, respectively, Proposition~\ref{equation} implies
${\cal L}_{b}(b^2 \, h) - \inf h \ge 0$ and ${\cal L}_{b}(b^2 \, h) -\sup h \le 0$ in $\O$, whence
$\inf h \le {\cal L}_{b}(b^2 \, h) \le \sup h$.
\end{remark}

In the construction of $J$, we will consider Equation~\eqref{aux-eq} with $b= q\, u$; to ensure its solvability, we confine $u$ within the set
\begin{equation*}
\Lambda_q := \Bigl\{ u \in H^1_0(\O) \, \bigl| \, q \, u \not= 0 \Bigr\} \, .
\end{equation*}
The set $\Lambda_q$ is the complement in $H^1_0(\O)$ of the kernel of the bounded and linear operator $u \in H^1_0(\O) \mapsto q \, u \in L^3(\Omega)$.
If $q$ vanishes at most on a set of measure zero, then $\Lambda_q = H^1_0(\O) \setminus\{0\}$. In general, $\Lambda_q$ satisfies the following properties.

\begin{proposition}\label{lambda-q}{\rm \cite[Proposition 2.4]{lp1}}~{}
\begin{enumerate}
\item[{\rm (a)}] $\Lambda_q$ is open in $H^1_0(\O)$ with $\partial \Lambda_q = \bigl\{ u \in H^1_0(\O) \, | \, q \, u = 0 \bigr\}$.
\item[{\rm (b)}] If $u\in H^1_0(\O)$ and $\dist(u,\partial \Lambda_q) \to 0$, then $\|q \, u\|_3 \to 0$.
\item[{\rm (c)}] $\Lambda_q$ contains subsets with arbitrarily large genus.
\end{enumerate}
\end{proposition}

\section{The constrained functional}\label{J-birth}

For $u \in \Lambda_q$, let
\begin{equation}\label{rho-u}
\rho_u := \dfrac{A}{|\O|} - (q \, u)^2 \, \chi-\o \, q \, u^2
\end{equation}
and consider  Equation~\eqref{aux-eq} with $b= q\, u$ and $\rho=\rho_u$. By Proposition~\ref{equation}, the associated homogeneous Neumann problem has a unique solution ${\cal L}_{qu}(\rho_u)$ in $H^1(\O)$.

Let us define the map $\Phi:\Lambda_q \longrightarrow H^1(\O)$
by letting $\Phi(u):= {\cal L}_{qu}(\rho_u)$,
and the functional $J : \Lambda_q \longrightarrow \R$ by letting
$J(u) := F\bigl(u,\Phi(u)\bigr)$.

\begin{proposition}\label{CF}
\begin{enumerate}
\item[{\rm (a)}] The map $\Phi$ is continuously differentiable in $\Lambda_q$. The graph of $\Phi$ is the set
$\bigl\{ (u,\p) \in \Lambda_q \times H^1(\O) \, | \, F'_\p(u,\p)=0 \bigr\}$.
\item[{\rm (b)}] The functional $J$ is continuously differentiable in $\Lambda_q$. Furthermore,
$(u,\p) \in \Lambda_q \times H^1(\O)$ is a critical point of $F$ if, and only if, $u$ is a critical point of $J$ and $\p = \Phi(u)$.
\end{enumerate}
\end{proposition}

\begin{Proof}
For the proof of Part~(a), see~\cite[Section~3]{lp1}. Note that all the assertions remain true despite the additional term $-\o \, q \, u^2$ appearing in the right-hand side of~\eqref{rho-u} when $\o \ne 0$.
Part~(b) easily follows from Part~(a).
\end{Proof}

On account of Proposition~\ref{CF}, nontrivial solutions to Problem~\eqref{KGM}-\eqref{BC-no} are in one-to-one correspondence with critical points of $J$ in $\Lambda_q$.

\subsection{Decomposition of $J$}
To simplify the notation, let $\p_u:= \Phi(u)$. Since $\p_u$ solves the homogeneous Neumann problem associated with the equation
\begin{equation*}
- \Delta \p + (q \, u)^2 \, \p =  \dfrac{A}{|\O|} - (q \, u)^2 \, \chi\, - \, \o \, q \, u^2\, ,
\end{equation*}
we get
\begin{equation*}
\|\nabla \p_u\|_2^2 =  A  \, \overline \p_u  - \int_\O (q \, u)^2 \, \chi \, \p_u \, dx - \int_\O \o \, q\, u^2 \, \p_u \, dx - \int_\O (q \, u)^2 \, \p_u^2 \, dx  \, ,
\end{equation*}
and thus,
\begin{equation*}\label{J-one}
\begin{split}
J(u)  = F\bigl(u,\p_u\bigr)  &=
\|\nabla u\|_2^2 + \int_\O \bigl(m^2 - (\o + q \, \chi)^2 \bigr) u^2 \, dx + {} \\[2mm]
&  - \int_\O (q\, u)^2 \, \chi \, \p_u \, dx -   \int_\O \o \, q\, u^2 \, \p_u \, dx \, +  A  \, \overline \p_u  .
\end{split}
\end{equation*}

For every $u \in \Lambda_q$, let
\begin{equation*}
\xi_u := -{\cal L}_{qu} ((q \, u)^2 \, \chi)  \, , \quad \eta_u := \frac{A}{|\O|} \, {\cal L}_{qu} (1)  \, , \quad \theta_u := -  {\cal L}_{qu} (q \, u^2) \, .
\end{equation*}
Note that $\eta_u$, $\xi_u$, and $\theta_u$ satisfy the equations
\begin{align}
- \D \xi_u + (q \, u)^2 \, \xi_u &= - \, (q \, u)^2 \, \chi  \, , \label{xi-u} \\[2mm]
- \D \eta_u + (q \, u)^2 \, \eta_u  &=  \dfrac{A}{|\O|} \, , \label{eta-u} \\[1mm]
- \D \theta_u + (q \, u)^2 \, \theta_u &= - \, q \, u^2 \, , \label{theta-u}
\end{align}
respectively, with homogeneous Neumann boundary conditions, and
\begin{equation}\label{decomp}
\p_u = \xi_u + \eta_u + \o \, \theta_u \, .
\end{equation}

We will write the functional $J$ in terms of $u$, $\xi_u$, $\eta_u$, and $\theta_u$.
Observe that
\begin{gather}
\int_\O (q\, u)^2 \, \chi \, \theta_u \, dx = \int_\O q \, u^2 \, \xi_u \, dx \, , \qquad A \, \ovxi  = - \int_\O (q\, u)^2 \, \chi \, \eta_u \, dx \, , \label{MIX2} \\
 A \, \ovtheta   = - \int_\O q \, u^2 \, \eta_u \, dx \, , \label{MIX1}
\end{gather}
these equalities are easily obtained by multiplying each of the equations~\eqref{xi-u}-\eqref{theta-u} by the solution of the remaining two equations. Taking~\eqref{decomp}--\eqref{MIX1} into account yields
\begin{equation}\label{J-dec}
J(u)  = \widetilde J(u) +  A \, \oveta + 2 \, \o \, A  \, \ovtheta
\end{equation}
for every $u\in \Lambda_q$, with
\begin{equation*}
\begin{split}
\widetilde J(u)
& = \|\nabla u\|_2^2 + \int_\O \bigl(m^2 - \o^2) \, u^2 \, dx  -  \int_\O 2 \, \o \, q  \, u^2 \, (\chi+\xi_u) \, dx  +  {} \\[2mm]
&
+ 2 \, A \, \ovxi - \int_\O (q\, u)^2 \, \chi^2 \, dx   - \int_\O (q \, u)^2 \, \chi \, \xi_u \, dx  - \int_\O \o^2 \, q \, u^2 \, \theta_u \, dx \, .
\end{split}
\end{equation*}

\subsection{Properties of $\xi_u$, $\eta_u$,  and $\theta_u$}

Since $\xi_u := -{\cal L}_{qu} ((q \, u)^2 \, \chi)$,
by Remark~\ref{uniform-bound} we have
\begin{equation}\label{xi}
\|\xi_u \|_{\infty} \le \|\chi\|_\infty
\end{equation}
for every $u \in \Lambda_q$.

\begin{lemma}\label{eta} {\rm \cite[Lemma 3.3]{lp1}}~{}
\begin{itemize}
\item[{\rm (a)}] For every $u \in \Lambda_q$, $A \, \eta_u \ge 0$ in $\O$.
\item[{\rm (b)}] There exists $\gamma \in(0,\infty)$ such that
$\|\nabla \eta_u\|_2 \le \gamma \, \|q\, u\|_3^2 \; | \overline \eta_u |$  for every $u \in \Lambda_q$.
\item[{\rm (c)}] Suppose $A\ne 0$. If $u\in \Lambda_q$ and $\|q \, u\|_3 \to 0$, then $A \, \overline \eta_u \to \infty$.
\end{itemize}
\end{lemma}

\begin{lemma}\label{theta-bounded}
Suppose that assumption {\rm (Q)} is satisfied. Then, for every $u \in \Lambda_q$, we have $\|\theta_u \|_\infty \le {1}/{q_0}$.
\end{lemma}

\begin{Proof}
Fix $u \in \Lambda_q$. Define $h \in L^\infty(\O)$ by
\begin{equation*}
h(x) := \begin{cases} {1}/{q(x)}  & \hbox{if $|q(x)|\ge q_0$,} \\[-1mm] 0 & \hbox{otherwise.}  \end{cases}
\end{equation*}
In view of assumption (Q), we have $q = q^2 \, h$ in $\O$, hence
 $\theta_u := -  {\cal L}_{qu} (q \, u^2) = - {\cal L}_{qu} ((q \, u)^2 \, h)$. Then $\|\theta_u\|_\infty \le \|h\|_\infty$, by Remark~\ref{uniform-bound}, and the conclusion readily follows.
\end{Proof}

\begin{remark}\label{noQ}
\begin{comment}
\end{comment}
If assumption (Q) is satisfied, the map $u \in \Lambda_q \mapsto 2 \, \o \,  A \, \ovtheta$, which appears as the third summand in~\eqref{J-dec},
is bounded from below. Without assumption (Q), this need not be the case.

For instance, suppose that $q \in C(\O) \cap L^6(\O) \setminus\{0\}$ does not satisfy~(Q). Hence, either $\inf\, \{q(x) \, | \, q(x)>0\}= 0$ or $\, \sup\, \{q(x) \, | \, q(x)<0\}= 0$. In the former case, let $\{s_n\}$ be any unbounded increasing sequence. Up to a subsequence, the open set
$$\O_n^+ := \left\{x \in \O \, | \, s_{n} < 1/q(x) < s_{n+1} \right\}$$
is nonempty. Take $u_n \in C_0^\infty(\O_n^+)\setminus\{0\} \subset \Lambda_q$.
Define $h_n : \O \to \R$ by
\begin{equation*}
h_n(x) := \begin{cases} 1/{q(x)}  & \hbox{if $x \in \O_n^+$,} \\[-1mm] s_n & \hbox{otherwise;}  \end{cases}
\end{equation*}
clearly, $s_n \le h_n < s_{n+1}$ in $\O$.
Since $q \, u_n^2 = (q\, u_n)^2 h_n$ in $\O$, we have
$\theta_{u_n} := -{\cal L}_{qu_n} (q \, u_n^2) = -{\cal L}_{qu_n} ((q\, u_n)^2 h_n)$.  By Remark~\ref{uniform-bound}, we get $\theta_{u_n} \le -s_n$ in $\O$ and thus,
$\overline \theta_{u_n} \to - \infty$.
Likewise, in the case $\, \sup\, \{q(x) \, | \, q(x)<0\}= 0$, we find a sequence $\{u_n\}\subset \Lambda_q$ such that $\overline \theta_{u_n} \to \infty$.
Therefore, depending on the sign of~$\o$ and~$A$, the map $u \in \Lambda_q \mapsto 2 \, \o \,  A \, \ovtheta$ may be unbounded from below.
\end{remark}

\begin{lemma}\label{theta} For every $u \in \Lambda_q$,
$\displaystyle \Bigl| 2 \, \o \,  A \, \ovtheta \Bigr| \le  \int_\O \o^2 \, u^2 \, dx + A \, \oveta$.
\end{lemma}

\begin{proof}
Fix $u \in \Lambda_q$. By~\eqref{MIX1},
\begin{equation*}
\Bigl| 2 \, \o \,  A  \, \ovtheta \Bigr|  = \left|\int_\O 2 \, \o \, q \, u^2 \, \eta_u \, dx  \right| \le \int_\O \o^2 \, u^2 \, dx +  \int_\O (q \, u)^2 \, \eta_u^2 \, dx  \, .
\end{equation*}
Multiplying~\eqref{eta-u} by $\eta_u$ yields
\begin{equation*}
\|\nabla \eta_u\|_2^2 + \int_\O (q\, u)^2 \, \eta_u^2 \, dx = A \, \oveta \, .
\end{equation*}
The conclusion readily follows.
\end{proof}

\begin{lemma}\label{eta-theta}
Let $\{u_n\} \subset \Lambda_q$ be bounded.
\begin{itemize}
\item[{\rm (a)}]  Suppose that $\{\|q\, u_n\|_3\}$ is bounded away from $0$. Then: up to a subsequence, $\{\eta_{u_n}\}$ and $\{\theta_{u_n}\}$ converge in $H^1(\O)$.
\item[{\rm (b)}]  If $A \ne 0$ and  $\|q \, u_n\|_3 \to 0$, then $ A \, \overline \eta_{u_n} + 2 \, \o \,  A \, \overline \theta_{u_n}  \to \infty$.
\end{itemize}
\end{lemma}

\begin{proof}
(a)\ Let $\{u_n\}\subset \Lambda_q$ be bounded.
Up to a subsequence, $\{u_n\}$ has in $L^6(\O)$ a limit $u$. Since $q \, u_n \to q\, u$ in $L^3(\O)$ and $\{\|q\, u_n\|_3\}$ is bounded away from $0$, we deduce $q \, u \ne 0$; moreover, $q \, u_n^2 \to q\, u^2$ in $L^{6/5}(\O)$.
Recall that $\eta_{u_n}:= {\cal L}_{qu_n} ({A}/{|\O|})$ and $\theta_{u_n}:= - {\cal L}_{qu_n} (q \, u_n^2)$.
Thus, by Proposition~\ref{equation}, $\eta_{u_n}$ and $\theta_{u_n}$ converge in $H^1(\O)$ to
 ${\cal L}_{q u} ({A}/{|\O|})$ and $-{\cal L}_{qu} (q \, u^2)$, respectively.
\\
(b)\ Preliminarily, fix $u \in \Lambda_q$ and note that,
by Lemma~\ref{eta}(b),
\begin{equation*}
\|\eta_u\|^2 = \|\nabla \eta_u\|_2^2 + |\oveta|^2 \le
\Bigl(\gamma^2 \, \|q \, u\|_3^4 + 1\Bigr)\, |\oveta|^2 \, .
\end{equation*}
Thus, by~\eqref{MIX1}, H\"older's inequality, and Sobolev's embedding theorem,
\begin{align*}
\Bigl| \o \, A \, \ovtheta \Bigr|  & = \Bigl| \int_\O \o \, q \, u^2 \, \eta_u \, dx \Bigr| \le |\o| \, \|q \, u\|_3 \, \|u\|_3 \, \|\eta_u\|_3 \\ &
 \le \sigma_3 \, \tau_3 \, |\o| \, \|q \, u\|_3 \, \|\nabla u\|_2 \, \|\eta_u\| \\ &
 \le \sigma_3 \, \tau_3 \,  |\o| \, \|q \, u\|_3 \,  \|\nabla u\|_2 \, \Bigl(\gamma^2 \, \|q \, u\|_3^4 + 1\Bigr)^{1/2}\, |\oveta| \, . \end{align*}
Therefore,
\begin{equation}\label{AT(u)-bis}
A \, \oveta + 2 \, \o \, A \, \ovtheta   \ge
|A| \, |\oveta| - \Bigl| 2 \, \o \, A \, \ovtheta \Bigr|
\ge \bigl(|A| - N(u)\bigr) \, |\oveta| \, ,
\end{equation}
with
\begin{equation*}
N(u):= 2 \, \sigma_3 \, \tau_3 \,  |\o| \, \|q \, u\|_3 \,  \|\nabla u\|_2 \, \Bigl(\gamma^2 \, \|q \, u\|_3^4 + 1\Bigr)^{1/2} \, .
\end{equation*}
Now assume that $\{u_n\} \subset \Lambda_q$ is bounded and $\|q \, u_n\|_3 \to 0$. Then $N(u_n) \to 0$, and the conclusion follows from~\eqref{AT(u)-bis} and Lemma~\ref{eta}(c).
\end{proof}

\subsection{Properties of $\widetilde J$}

\begin{lemma}\label{bound-1}
There exist $C_1 \in \R$, $C_2 \in (0,\infty)$, and $C_3 \in [0,\infty)$, which depend on $\O$, $\o$, and the norms $\|q\|_6$ and $\|\al\|_{1/2}$, such that
\begin{equation}\label{lowerJ}
   \widetilde J(u)   \ge C_1 \, \|\nabla u\|_2^2 + \int_\O \bigl(m^2 - \o^2\bigr) u^2 \, dx   - 2\, \kappa \, |A| \, \|\al\|_{1/2}
\end{equation}
and
\begin{equation}\label{upperJ}
\widetilde J(u)   \le \bigl(C_2 + C_3 \, \|\theta_u\|\bigr) \, \|\nabla u\|_2^2 + 2 \, \kappa \, |A|  \, \|\al\|_{1/2}
\end{equation}
for every $u\in \Lambda_q$.
\end{lemma}

\begin{proof}
Fix $u\in \Lambda_q$ and recall that
\begin{equation*}
\begin{split}
\widetilde J(u)
& = \|\nabla u\|_2^2 + \int_\O \bigl(m^2 - \o^2) \, u^2 \, dx  -  \int_\O 2 \, \o \, q  \, u^2 \, (\chi+\xi_u) \, dx  +  {} \\[2mm]
&
+ 2 \, A \, \ovxi - \int_\O (q\, u)^2 \, \chi^2 \, dx   - \int_\O (q \, u)^2 \, \chi \, \xi_u \, dx  - \int_\O \o^2 \, q \, u^2 \, \theta_u \, dx \, .
\end{split}
\end{equation*}
By H\"older's inequality, Sobolev's embedding theorem, \eqref{chi-infty}, and \eqref{xi},
\begin{equation}\label{bound-on-P1}
\begin{split}
\left| \int_\O  q \, u^2 \, (\chi + \xi_u) \, dx \right|  & \le
2 \, \|\chi\|_\infty \, \|q\|_6  \, \|u\|_{12/5}^2 \\ & \le 2 \, \kappa \, \sigma_{12/5}^2 \, \|\al\|_{1/2} \, \|q\|_6 \, \|\nabla u\|_2^2 \, ,
\end{split}
\end{equation}
\begin{equation}\label{bound-on-P2}
\int_\O (q\, u)^2 \, \chi^2 \, dx  \le
\|\chi\|_\infty^2 \|q\|_6^2  \, \|u\|_3^2  \le
 \kappa^2 \, \sigma_3^2 \, \|\al\|_{1/2}^2 \, \|q\|_6^2  \, \|\nabla u\|_2^2 \, ,
\end{equation}
and
\begin{equation}\label{co2}
|\ovxi|
\le \kappa \, \|\al\|_{1/2} \ .
\end{equation}
Multiplying~\eqref{xi-u} and~\eqref{theta-u} by $\xi_u$ and $\theta_u$, respectively, gives
\begin{equation}\label{positive}
- \int_\O (q \, u)^2 \, \chi \, \xi_u \, dx \ge 0 \, , \quad - \int_\O \o^2 \, q \, u^2 \, \theta_u \, dx \ge 0 \, .
\end{equation}
Taking~\eqref{bound-on-P1}-\eqref{positive} into account yields~\eqref{lowerJ}, with
\begin{equation}\label{C1}
C_1 := 1 - 4 \, |\o| \, \kappa \, \sigma_{12/5}^2 \, \|\al\|_{1/2} \, \|q\|_6 - \kappa^2 \, \sigma_3^2 \, \|\al\|_{1/2}^2 \, \|q\|_6^2 \, .
\end{equation}
To prove~\eqref{upperJ}, in addition to~\eqref{bound-on-P1}-\eqref{co2} observe that \begin{equation*}\label{up-1}
\left| \int_\O (q \, u)^2 \, \chi \, \xi_u \, dx \right| \le
\|\chi\|_\infty^2 \, \|q\|_6^2 \, \|u\|_3^2 \le
\kappa^2 \, \sigma_3^2 \, \|\al\|_{1/2}^2 \, \|q\|_6^2  \, \|\nabla u\|_2^2
\end{equation*}
and
\begin{equation*}\label{up-2}
\left| \int_\O q \, u^2 \, \theta_u \, dx \right| \le \|q\|_6 \, \|u\|_3^2 \, \|\theta_u\|_6 \le
\sigma_3^2 \,   _6 \, \|q\|_6 \, \|\nabla u\|_2^2 \,  \|\theta_u\| \, .
\end{equation*}
Thus, \eqref{upperJ} follows with
\begin{equation*}
C_2:= 1 + \bigl|m^2 - \o^2| \, \sigma_2^2
 + 4 \, |\o| \, \kappa \, \sigma_{12/5}^2 \, \|\al\|_{1/2}  \, \|q\|_6 +  2\, \kappa^2 \, \sigma_3^2 \, \|\al\|_{1/2}^2 \, \|q\|_6^2
\end{equation*}
and
$C_3:= \o^2 \, \sigma_3^2 \,  \tau_6 \, \|q\|_6$.
\end{proof}

\section{Properties of $J$}\label{prop-of-J}

Throughout this section, we will assume $A\ne 0$.
For ease of discussion, we will refer to
Case~1, if $|\o| \le |m|$ and (Q) is satisfied,  and to
Case~2, if $|\o| \le |m|/\sqrt 2$.
In either case, we will assume that $\|\al\|_{1/2} \, \|q\|_6$ is so small that the constant $C_1$, as defined in~\eqref{C1}, is strictly positive.

\begin{proposition}\label{coercive1}
The functional $J$ is bounded from below and coercive in $\Lambda_q$.
\end{proposition}

\begin{proof}
Fix $u \in \Lambda_q$. In Case 1, Lemma~\ref{theta-bounded} applies and implies
\begin{equation*}
A \, \oveta + 2 \, \o \,  A \, \ovtheta \ge - \dfrac{2 \, |A| \, |\o|}{q_0}\, ,
\end{equation*}
in view of  Lemma~\ref{eta}(a).
Therefore, \eqref{J-dec} and \eqref{lowerJ} yield
\begin{equation*}
J(u)
\ge
C_{1} \, \|\nabla u\|_2^2 + \int_\O \bigl(m^2 - \o^2\bigr) u^2 \, dx   - 2\, \kappa \, |A| \, \|\al\|_{1/2}
- \dfrac{2 \, |A| \, |\o|}{q_0}  \, .
\end{equation*}
In Case 2, Lemma~\ref{theta} implies
\[ A \, \oveta + 2 \, \o \,  A \, \ovtheta \ge - \int_\O \o^2 \, u^2 \, dx \, .\]
Therefore, \eqref{J-dec} and~\eqref{lowerJ} yield
\begin{equation*}
\begin{split}
J(u)
& \ge
C_{1} \, \|\nabla u\|_2^2 + \int_\O \bigl(m^2 - \o^2\bigr) u^2 \, dx   - 2\, \kappa \, |A| \, \|\al\|_{1/2}
- \int_\O \o^2 \, u^2 \, dx \\
& =
C_{1} \, \|\nabla u\|_2^2 + \int_\O \bigl(m^2 - 2 \, \o^2\bigr) u^2 \, dx   - 2\, \kappa \, |A| \, \|\al\|_{1/2} \, .
\end{split}
\end{equation*}
With $C_{1} \in (0,\infty)$, the conclusions readily follow in both cases.
\end{proof}

\begin{proposition}\label{boundary}~{}
\begin{itemize}
\item[{\rm (a)}] Every sequence $\{u_n\}\subset \Lambda_q$ such that $\|q \, u_n\|_3 \to 0$ has a subsequence $\{u_{k_n}\}$ such that $J(u_{k_n}) \to \infty$.
\item[{\rm (b)}] $J$ has complete sublevels.
\end{itemize}
\end{proposition}

\begin{proof}
(a)\ Let $\{u_n\}\subset \Lambda_q$ and assume $\|q\, u_n\|_3 \to 0$.\\
In Case 1, note that $A\, \overline \eta_{u_n} \to \infty$, by Lemma~\ref{eta}(c),
whereas $\{\overline \theta_{u_n}\}$ and $\{\widetilde J(u_n)\}$  are bounded from below, by  Lemma~\ref{theta-bounded} and~\eqref{lowerJ}.
Since $J(u_n) = \widetilde J(u_n) +  A \, \overline \eta_{u_n} + 2 \, \o \,  A \, \overline \theta_{u_n}$, we deduce $J(u_n) \to \infty$.  \\
In Case 2, we consider two possibilities.
If $\{\|\nabla u_n\|_2\}$ is unbounded, there exists a subsequence $\{u_{k_n}\}$ such that $\|\nabla u_{k_n}\|_2 \to \infty$; thus, \hbox{$J(u_{k_n}) \to \infty$}, for $J$ is coercive by Proposition~\ref{coercive1}.
If $\{\|\nabla u_n\|_2\}$ is bounded, then
$ A \, \overline \eta_{u_n} + 2 \, \o \,  A \, \overline \theta_{u_n}  \to \infty$, by Lemma~\ref{eta-theta}(b), whereas $\{\widetilde J(u_n)\}$ is bounded from below, by~\eqref{lowerJ}. Since $J(u_n) = \widetilde J(u_n) +  A \, \overline \eta_{u_n} + 2 \, \o \,  A \, \overline \theta_{u_n}$, we deduce $J(u_n) \to \infty$.
\\
(b)\ Suppose that $\{u_n\}\subset \Lambda_q$, $J(u_n)\le c$, for some $c \in \R$, and $u_n \to u$ in $H^1_0(\O)$.
By Part~(a), the sequence $\{\|q \, u_n\|_3\}$ is bounded away from $0$ and thus, by Proposition~\ref{lambda-q}(b), $u \in \Lambda_q$.
\end{proof}

\begin{proposition}\label{if-and-only-if}
For any $\{u_n\} \subset \Lambda_q$, the sequence $\{J(u_n)\}$ is unbounded if, and only if, either $\{u_n\}$ is unbounded or $\{\|q\, u_n\|_3\}$ is not bounded away from $0$.
\end{proposition}

\begin{proof}
The ``if'' part of the statement easily follows from Proposition~\ref{coercive1} and Proposition~\ref{boundary}. We will prove the ``only if'' part by way of contradiction.
Suppose that there exists a bounded sequence $\{u_n\}\subset \Lambda_q$ such that $\{\|q \, u_n\|_3\}$ is bounded away from $0$ and {$J(u_n)\to \infty$}. \\
Up to a subsequence, $\{\eta_{u_n}\}$ and $\{\theta_{u_n}\}$ converge in $H^1(\O)$, in view of Lemma~\ref{eta-theta}(a). This clearly implies that
$\{\overline{\eta}_{u_n}\}$ and $\{\overline{\theta}_{u_n}\}$ are bounded and thus, $\{\widetilde J(u_n)\}$ is bounded from above, by~\eqref{upperJ}.
Since
$J(u_n)  = \widetilde J(u_n) +  A \, \overline \eta_{u_n} + 2 \, \o \,  A \, \overline \theta_{u_n}$,
we deduce that $\{J(u_n)\}$ is bounded from above, a contradiction.
\end{proof}

\smallskip
\begin{proposition}\label{P-S}
The functional $J$ satisfies the Palais-Smale condition in $\Lambda_q$.
\end{proposition}

\begin{proof}
Suppose that $\{u_n\} \subset \Lambda_q$ is a Palais-Smale sequence, that is, $\{J(u_n)\}$ is bounded and $J'(u_n) \to 0$; we have to show that, up to a subsequence, $\{u_n\}$ converges in $\Lambda_q$. \\
Since $J$ is coercive, $\{u_n\}$ is bounded in $H^1_0(\O)$; up to a subsequence, it converges weakly to some $u\in H^1_0(\O)$. Observe that
\begin{equation}\label{PS}
\D u_n  = - \dfrac{1}{2} \, J'(u_n) +   m^2 \, u_n - \bigl( \o + q\, (\eta_{u_n} + \xi_{u_n} + \o \, \theta_{u_n} + \chi)\bigr)^2 \,  u_n \, .
\end{equation}
The first two summands in the right-hand side of~\eqref{PS} are clearly bounded in $H^{-1}(\O)$; we will show that the same is true for the third summand.\\
Since $\{J(u_n)\}$ is bounded, Proposition~\ref{boundary} implies that $\{\|q\, u_n\|_3\}$ is bounded away from $0$. Lemma~\ref{eta-theta}(a) applies: up to a subsequence, $\{\eta_{u_n}\}$ and $\{\theta_{u_n}\}$ converge in $H^1(\O)$, and are therefore bounded in $L^6(\O)$.
By~\eqref{xi}, $\{\xi_{u_n}+\chi\}$ is bounded in $L^6(\O)$ as well. It follows that $\{(\eta_{u_n} + \xi_{u_n} + \o \, \theta_{u_n} + \chi)^2\}$ is bounded in $L^3(\O)$, which in turn implies that $\bigl\{q^2 \, (\eta_{u_n} + \xi_{u_n} + \o \, \theta_{u_n} + \chi)^2 \, u_n \bigr\}$ is bounded in $L^{6/5}(\O)$, hence in $H^{-1}(\O)$. \\
On account of~\eqref{PS}, the sequence $\{\Delta u_n\}$ is bounded in $H^{-1}(\O)$;
the compactness of the inverse Laplace operator implies that, up to a subsequence, $\{u_n\}$ converges to $u$ in $H^1_0(\O)$. By Proposition~\ref{boundary}(b), $u \in \Lambda_q$.
\end{proof}

\section{Proof of the main results}\label{proof-of-main}

\begin{Proof}[Proof of Theorems~\ref{main1} and~\ref{main2}]
On account of the correspondence between critical points of $J$ and nontrivial solutions to Problem~\eqref{KGM}-\eqref{BC-no},
it suffices to prove that $J$ has a sequence of critical points $\{u_n\} \subset \Lambda_q$ satisfying~(i) and~(ii).
Observe that $J(u)=J(|u|)$ for every $u \in \Lambda_q$. This easily follows from the fact that $\Phi(u)=\Phi(|u|)$ for every $u \in \Lambda_q$, by the very definition of $\Phi$.

Suppose that $A\ne 0$
and $\|\al\|_{1/2} \, \|q\|_6$ is so small that the constant $C_1$, as defined in~\eqref{C1}, is strictly positive.
As we have shown in Section~\ref{prop-of-J}, under the assumptions in Theorems~\ref{main1} and~\ref{main2}, the functional $J$ is bounded from below, has complete sublevels, and
satisfies the Palais-Smale condition in $\Lambda_q$. These properties readily imply that $J$ attains its minimum at some $u_0 \in \Lambda_q$; by the observation above,  we can assume $u_0 \ge 0$ in $\O$.

By Proposition~\ref{lambda-q}(c), the set $\Lambda_q$ has infinite genus. Thus, Ljuster\-nik-Schnirel\-mann Theory applies (see~\cite[Corollary 4.1]{szulkin} and \cite[Remark~3.6]{ACZ}) and $J$ has a sequence $\{u_n\}_{n \ge 1}$ of critical points in $\Lambda_q$. Standard arguments show that $J(u_n)\to\infty$ (see~\cite[Chapter 10]{AM}).

Let $\{v_{n}\}$ be a bounded subsequence of $\{u_n\}$. In view of Proposition~\ref{if-and-only-if}, every subsequence of $\{v_{n}\}$ has a subsequence $\{v_{k_n}\}$ such that $\|q\, v_{k_n}\|_3\to 0$; this proves that $\|q \, v_n\|_3 \to 0$.
\end{Proof}

\begin{Proof}[Proof of Theorem~\ref{main3}]

Assume that $|\o| \le |m|$ and
$\|\al\|_{1/2} \, \|q\|_6$ is so small that the constant $C_1$, as defined in~\eqref{C1}, is strictly positive. \\
Suppose that $(u,\phi)$ is a solution to~\eqref{KGM}-\eqref{BC-no} with $A=0$ and let $\p:=\phi-\chi$. Then, $(u,\p)$ is a solution to
\begin{equation}\label{PROB-0}
\begin{cases}
\Delta u = m^2 u - \bigl(\o + q \, (\p+\chi)\bigr)^2 \, u \quad & \hbox{in $\O$,} \\[1mm]
\Delta \p =  q \, \bigl(\o + q \, (\p+\chi)\bigr) \, u^2 & \hbox{in $\O$,} \\[1mm]
\hskip 3.6mm u = \dfrac{\partial \p}{\partial \n} = 0 & \hbox{on $\partial \O$.}
\end{cases}
\end{equation}
Multiplying by $u$ the first equation in~\eqref{PROB-0} gives
\begin{align*}
0 & = \|\nabla u\|_2^2  + \int_\O m^2 \, u^2 \, dx - \int_\O \bigl(\o + q \, (\p+\chi)\bigr)^2 \, u^2 \, dx  \notag \\
& = \|\nabla u\|_2^2  + \int_\O (m^2-\o^2) \, u^2 \, dx - \int_\O (q\, u)^2 \, \p^2 \, dx - \int_\O 2 \, \o \, q \, u^2 \, \chi  \, dx   + {} \notag \\
& - \int_\O (q\, u)^2 \, \chi^2 \, dx   - \int_\O 2 \, (q\, u)^2 \, \p \, \chi \, dx - \int_\O 2 \, \o \, q \, u^2 \, \p  \, dx \, .
\end{align*}
Multiplying by $\p$ the second equation in~\eqref{PROB-0} gives
\begin{equation}\label{phi-1}
\|\nabla \p\|_2^2 + \int_\O (q\, u)^2 \, \p^2 \, dx  = - \int_\O (q\, u)^2 \, \p \, \chi \, dx - \int_\O \o \, q \, u^2 \p \, dx  \, .
\end{equation}
Substituting~\eqref{phi-1} into the preceding equality gives
\begin{align*}
0
 = \|\nabla u\|_2^2 & + \int_\O (m^2-\o^2) \, u^2 \, dx + \int_\O (q\, u)^2 \, \p^2 \, dx  - \int_\O 2 \, \o \, q \, u^2 \, \chi  \, dx   + {} \notag \\
& - \int_\O (q\, u)^2 \, \chi^2 \, dx  + 2\, \|\nabla \p\|_2^2 \, .
\end{align*}
Neglecting the nonnegative terms, and recalling~\eqref{bound-on-P1}, \eqref{bound-on-P2}, and the definition of~$C_1$, we obtain
\begin{align*}
0  & \ge \|\nabla u\|_2^2   - \int_\O 2 \, \o \, q \, u^2 \, \chi  \, dx - \int_\O (q\, u)^2 \, \chi^2 \, dx  \notag \\
& \ge \Bigl[1 -
2 \, |\o| \, \kappa \, \sigma_{12/5}^2 \, \|\al\|_{1/2} \, \|q\|_6 -  \kappa^2 \, \sigma_3^2 \, \|\al\|_{1/2}^2 \, \|q\|_6^2\Bigr] \, \|\nabla u\|_2^2 \\ & \ge C_1 \, \|\nabla u\|_2^2 \, ,
\end{align*}
which implies $u=0$.
\end{Proof}

\end{document}